\documentclass[10pt]{article}

\usepackage{amssymb,amsthm,amsmath}
\usepackage{enumerate}
\usepackage{graphicx,color}
\usepackage{setspace}
\usepackage[hidelinks]{hyperref}

\newcommand{\dd}{\mathrm{d}}
\newcommand{\E}{\mathbb{E}}

\newcommand{\R}{\mathbb{R}}
\newcommand{\C}{\mathbb{C}}
\newcommand{\p}[1]{\mathbb{P}\left( #1 \right)}
\newcommand{\scal}[2]{\left\langle #1, #2 \right\rangle}

\DeclareMathOperator{\vol}{vol}

\newtheorem{theorem}{Theorem}
\newtheorem{lemma}[theorem]{Lemma}

\theoremstyle{remark}

\theoremstyle{definition}

\title{\vspace{-3em}Szarek and Ball made simple}
\author{Colin Tang$^\dagger$ and Tomasz Tkocz$^\dagger$\footnote{Email: ttkocz@math.cmu.edu. 
Research supported in part by 
Simons Foundation International Travel Support for Mathematicians grant SFI-MPS-TSM-00025936.}
}

\date{\begin{normalsize}
\emph{$^\dagger$Department of Mathematical Sciences, Carnegie Mellon University, Pittsburgh, PA 15213, USA} 
\end{normalsize}}

\begin{document}

\maketitle

\begin{abstract}
We present proofs of Szarek's moment comparison and Ball's cube slicing inequalities based on elementary convexity and majorisation. The key ingredient is a moment-lifting argument which allows us to use Schur monotonicity, avoiding the integral inequalities used in the original Fourier-analytic approaches.
\end{abstract}

\bigskip

\begin{footnotesize}

\end{footnotesize}

\bigskip

\section{Introduction}

Szarek's $L_1-L_2$ moment comparison inequality \cite{Sz} for Rademacher sums $X$ states that
\begin{equation}\label{eq:Szarek}
\E|X| \geq \frac{1}{\sqrt{2}}
\end{equation}
for \emph{all} random variables $X$ of the form $X = \sum_{j=1}^n a_j\varepsilon_j$, where $\varepsilon_1, \varepsilon_2, \dots$ are independent and identically distributed (i.i.d.) Rademacher random variables, that is, random signs, $\p{\varepsilon_j = \pm 1} = \frac12$, and $a_1, \dots, a_n$ are real numbers satisfying
\begin{equation}\label{eq:L2-const}
\E|X|^2 = \sum_{j=1}^{n} a_j^2 = 1.
 \end{equation}
Moreover, when $n \geq 2$ the inequality is sharp and equality holds if and only if the coefficient vector $a = (a_1, \dots, a_n)$ is $a = (\frac{1}{\sqrt2}, \frac{1}{\sqrt2},0,\dots, 0)$, up to permutations and sign changes.

Ball's cube slicing inequality \cite{Ball-cube} is a geometric result asserting that the maximal-volume central hyperplane sections $[-\frac12, \frac12]^n \cap a^\perp$ of the unit-volume cube $[-\frac12, \frac12]^n$ in $\R^n$ occur precisely at the two-coordinate diagonal sections by hyperplanes $a^\perp$ with unit normal vectors $a = (a_1, \dots, a_n)$ proportional to $(1, 1, 0, \dots, 0)$, again, up to symmetries: permutations and coordinate-hyperplane reflections.

Both inequalities admit arguably natural geometric and probabilistic formulations. Szarek's inequality is equivalent to the statement, \emph{dual} to cube slicing, about minimal-volume hyperplane projections of the cross-polytope $B_1^n = \{x \in \R^n, \ \sum_{j=1}^n |x_j| \leq 1\}$ (see \cite{Ball-wave, NT-surv}), with the same extremising hyperplanes, whereas Ball's inequality is equivalent to a probabilistic statement about sums of independent uniform random variables (see \cite{Ball-wave}). Thanks to Archimedes' hat-box theorem (see \cite{CKT, NT-surv}), this can be further recast in perfect analogy to \eqref{eq:Szarek}, namely Ball's cube slicing inequality states that
\begin{equation}\label{eq:Ball}
\E\left[|X|^{-1}\right] \leq \sqrt{2}
\end{equation}
for \emph{all} random vectors $X$ of the form $X = \sum_{j=1}^n a_j\xi_j$, where $\xi_1, \xi_2, \dots$ are i.i.d. uniform on the unit Euclidean sphere $S^2$ in $\R^3$ and $a_1, \dots, a_n$ are real numbers satisfying the same constraint \eqref{eq:L2-const}. We shall refer to such a random vector $X$ as a \emph{spherical sum}, to parallel the term ``Rademacher sum''.
Here and throughout, $|\cdot|$ also denotes the standard Euclidean norm.

Each of these results has been quite influential in its own and intertwined right, with a great deal of understanding of their various extensions and refinements (see, e.g. \cite{NT-surv}). Two particularly effective approaches use Fourier integral estimates (Haagerup's approach from \cite{Haa} to \eqref{eq:Szarek} and Ball's original proof of \eqref{eq:Ball} in \cite{Ball-cube}). Additionally, in the Rademacher case, there is a one of its kind proof from the book, a spectral-gap argument of Lata\l a and Oleszkiewicz \cite{LO} using the discrete Fourier analytic framework (see, e.g. \cite{NT-surv}, p. 381, or Exercise~2.55 in \cite{ODonell-book} for a  modern presentation), taking \eqref{eq:Szarek} beyond the Euclidean setting to every normed space. We give a brief historical account in the final section.

Our purpose here is to show a new non-technical convexity-based argument explaining both inequalities in a unified fashion, dispensing with the Fourier-methods. 
For brevity, we focus on the inequalities and the sharpness of their constants, and do not pursue the characterisation of all equality cases in \eqref{eq:Szarek} and \eqref{eq:Ball}. This characterisation will follow from stability estimates in a companion paper \cite{TT}.

In the next section, we  recall the classical notion of majorisation and the well-known Schur-monotonicity results available for $L_p$ moments of sums $X$ featured in \eqref{eq:Szarek} and \eqref{eq:Ball}. They provide the very underpinning of a new argument making Szarek's and Ball's inequalities \emph{simple}. The argument is then presented in detail in the next two sections. The final section is devoted to a wider overview and a natural programme for further work.

\section{Majorisation background}

For two nonnegative vectors $x = (x_1, \dots, x_n)$ and $y = (y_1, \dots, y_n)$ in $\R_+^n$ with equal total mass $\sum_{j=1}^{n} x_j = \sum_{j=1}^n y_j$, we say that $x$ is \emph{majorised} by $y$, denoted $x \prec y$, if all partial sums of a nonincreasing rearrangement $x^* = (x_j^*)_{j=1}^n$ of $x$ (i.e. a permutation with $x_1^* \geq \dots \geq x_n^*$) do not exceed the corresponding partial sums of $y^*$,
\begin{equation}\label{eq:maj-def}
\sum_{j=1}^k x_j^* \leq \sum_{j=1}^k y_j^*, \qquad \text{for all} \quad k = 1, \dots, n.
 \end{equation}
There are several other equivalent definitions. One is that $x = Py$, for an $n \times n$ doubly stochastic matrix $P$. Alternatively, $x$ is obtained from $y$ by a finite sequence of $T$-transforms, where each is a convex combination of the identity and a transposition, namely of the form 
\begin{align*}
x_i &= (1-\lambda)y_i + \lambda y_j = y_i + \lambda(y_j-y_i), \\
x_j &= \lambda y_i + (1-\lambda)y_j = y_j - \lambda(y_j-y_i),
 \end{align*}
for some indices $i \neq j$ and $0 \leq \lambda \leq 1$, with the remaining coordinates unchanged
(robbing Peter to pay Paul).
Finally, geometrically, $x \prec y$ if $x$ lies in the convex hull of all permutations of $y$, $x = \sum_{\sigma} \lambda_\sigma y_\sigma$ for some nonnegative weights $\lambda_\sigma$ adding up to $1$, indexed by all permutations $\sigma$ of the $n$-element index set, and $y_\sigma = (y_{\sigma(1)}, \dots, y_{\sigma(n)})$. 

For example, among all nonnegative vectors $x$ of total mass $1$, the extremal ones are the constant vector and the standard basis vectors, that is we have
\begin{equation}\label{eq:maj-example}
\left(\frac1n, \dots, \frac1n\right) \ \prec \ x \ \prec \ (1, 0, \dots, 0).
\end{equation}

A function $f\colon \R_+^n \to \R$ is Schur-convex (concave) if it preserves (reverses) the majorisation order, that is $f(x) \leq f(y)$ ($f(x) \geq f(y)$) provided that $x \prec y$. We refer to \cite{Bhatia, MOA} for further background on this classical elementary notion.

As an instructive example, vital in the sequel, fix $x \in \R$ and consider the function
\[ 
h(t) = \E|x + \sqrt{t}\varepsilon|^3, \qquad t \geq 0,
 \]
where $\varepsilon$ is a Rademacher random variable.
Without loss of generality, we can assume that $x \geq 0$ (thanks to the symmetry of $\varepsilon$). Then, explicitly, 
\[ 
h(t) = \frac{|x+ \sqrt{t}|^3 + |x - \sqrt{t}|^3}{2} = \begin{cases}
x^3 + 3xt, & 0 \leq t \leq x^2, \\
t^{3/2} + 3x^2\sqrt{t}, & t \geq x^2.
\end{cases}
 \]
Suppose $x > 0$. The function $h$ is affine on $[0, x^2]$, the one-sided derivatives at $t = x^2$ match and $h''(t) = \frac{3(t-x^2)}{4t^{3/2}}$, $t > x^2$, so $h$ is strictly convex on $(x^2, \infty)$ and convex globally. At $x=0$, $h(t) = t^{3/2}$, which is plainly convex. Consequently, we have the following claim,
\begin{equation}\label{eq:L3-Schur}
F(x_1, \dots, x_n) = \E\Big[|\sqrt{x_1}\varepsilon_1 + \dots + \sqrt{x_n}\varepsilon_n|^3\Big] \qquad \text{is Schur-concave,}
\end{equation}
which is classical, going back to Eaton's criterion \cite{Eat}, verified for power functions $|\cdot|^p$, $p \geq 3$ by Komorowski in \cite{Kom}. 

For completeness, we present a direct argument. To see \eqref{eq:L3-Schur}, it suffices to verify this under a single $T$-transform. Thanks to symmetry under permutations, this in turn is reduced to verifying that the map
\[ 
f(\delta) = F(a+\delta, a-\delta, x_3, \dots, x_n)     
 \]
is nonincreasing on $[0, a]$ for a fixed $a > 0$ and $x_3, \dots, x_n \geq 0$. Conditioning on $(\varepsilon_j)_{j \geq 3}$ and letting $x = \sum_{j \geq 3} \sqrt{x_j}\varepsilon_j$, we have
\begin{align*}
f(\delta) &= \E_x\E_{\varepsilon_1, \varepsilon_2}\left[\Big|\sqrt{a+\delta}\varepsilon_1 + \sqrt{a-\delta}\varepsilon_2 + x\Big|^3\right] \\
&= \E_x\E_{\varepsilon_1, \varepsilon_2} h_x\left(|\sqrt{a+\delta}\varepsilon_1 + \sqrt{a-\delta}\varepsilon_2|^2\right) \\
&= \E_x\E_{\varepsilon} h_x(2a + 2\sqrt{a^2-\delta^2}\varepsilon),
 \end{align*}
where we denote $h_x(t) = h(t)$ to emphasise the dependence on $x$.
The point is that since $t \mapsto h_x(t)$ is convex, the function
\[ 
t \mapsto \E_\varepsilon h_x(2a(1+t\varepsilon)), \qquad -1 \leq t \leq 1
 \]
is convex and even, hence nondecreasing on $[0,1]$, proving the claimed Schur-concavity \eqref{eq:L3-Schur}.

Combined with \eqref{eq:maj-example}, \eqref{eq:L3-Schur} gives that every Rademacher sum $X = \sum_{j=1}^n a_j\varepsilon_j$ of length at most $n$, subject to the $L_2$ constraint \eqref{eq:L2-const}, satisfies
\[ 
1 \leq \E|X|^3 \leq \E\left|\frac{\varepsilon_1 + \dots + \varepsilon_n}{\sqrt{n}}\right|^3
 \]
(nota bene, also by \eqref{eq:L3-Schur}, the right-hand side is nondecreasing in $n$).

\section{Szarek's inequality made simple}

As in the Fourier proofs \cite{Ball-cube, Haa}, we distinguish the \emph{easy} case, when one coefficient carries at least half of the total squared mass, from the opposite \emph{hard} case, when every coefficient carries less than half.

\subsection[Easy case]{\emph{Easy} case: $\max_j |a_j| \geq \frac{1}{\sqrt2}$.}

For the proof of \eqref{eq:Szarek}, first note the elementary identity
\begin{equation}\label{eq:L_1-n=2}
\E|a\varepsilon_1 + b\varepsilon_2| = \max\{|a|,|b|\}, \qquad a, b \in \R.
\end{equation}
In particular, for the Rademacher sum $X = \sum_{j=1}^{n} a_j\varepsilon_j$, we have
\[ 
\E|X| \geq \max_j |a_j|.
 \]
Consequently, \eqref{eq:Szarek} holds whenever 
\[
\max_j |a_j| \geq \frac{1}{\sqrt2}.
 \]

\subsection[Hard case]{\emph{Hard} case: $\max_j |a_j| \leq \frac{1}{\sqrt2}$.}

Setting $x_j = a_j^2$, the opposite case is equivalent to
\begin{equation}\label{eq:Szarek-hard-case}
\E\left|\sum_{j=1}^n \sqrt{x_j}\varepsilon_j \right| \geq \frac{1}{\sqrt{2}} = \E\left|\frac{1}{\sqrt2}\varepsilon_1+\frac{1}{\sqrt2}\varepsilon_2\right|,
 \end{equation}
for all $x = (x_1, \dots, x_n) \in \R_+^n$ with $\sum x_j = 1$ and $x_j \leq \frac12$, for every $j$. For such vectors, 
\[ 
(x_1, \dots, x_n) \prec \left(\frac12, \frac12, 0, \dots, 0\right),
 \]
but the map $(x_1, \dots, x_n) \mapsto \E|\sum \sqrt{x_j}\varepsilon_j|$ fails to be Schur-monotone (otherwise, Szarek's inequality would plainly have a different extremiser). However, we have \eqref{eq:L3-Schur}, so it is a compelling idea to connect the $L_1$ and $L_3$ moments of $X$.  

\subsection[The L1-L3 bridge]{The $L_1-L_3$ bridge.}

A natural link between power functions with different exponents is offered by differentiation. If a function $f\colon \R \to \R$ is $C^2$, then for a fixed $x \in \R$, as $h \to 0$,
\[ 
\E f(x+ h \varepsilon) = \frac{f(x+h) + f(x-h)}{2} = f(x) + \frac{h^2}{2}f''(x) + o(h^2),
 \]
thus
\begin{equation}\label{eq:Rad-diff}
\frac{\dd }{\dd t}\Big|_{t=0+} \E f(x+\sqrt{t}\varepsilon) = \frac{1}{2}f''(x).
 \end{equation}
We emphasise that the square-root order of the increment is the right scaling for the Rademacher averaging.
In particular,
\begin{equation}\label{eq:Rad-diff-L3}
\frac{\dd }{\dd t}\Big|_{t=0+} \E |x+\sqrt{t}\varepsilon|^3 = 3|x|.
\end{equation}
We are ready to establish the heart of the whole argument, a sharp $L_1-L_3$ moment comparison which offers the needed bridge.

\begin{lemma}\label{lm:Rad-bridge}
For every Rademacher sum $X = \sum_{j=1}^{n} a_j\varepsilon_j$, $a_j \in \R$, we have
\begin{equation}\label{eq:Rad-bridge}
\E|X| \cdot \E|X|^2 \geq \frac{1}{2}\E|X|^3,
\end{equation}
which is sharp.
\end{lemma}
\begin{proof}
The assertion is trivial for $X = 0$. Otherwise, by homogeneity and after discarding zero coefficients and absorbing signs, we may write $X=\sum_{j=1}^n\sqrt{x_j}\varepsilon_j$, where $x_j>0$ and $\sum_jx_j=1$. To use \eqref{eq:Rad-diff-L3}, we introduce a new auxiliary summand $\sqrt{t}\varepsilon$ independent of $X$, with a small positive weight $t$ and \emph{spread} the weight out among the rest proportionally to the masses $x_j$. Specifically, note the following majorisation,
\[ 
(x_1, \dots, x_n, t) \prec \big((1+t)x_1, \dots, (1+t)x_n, 0\big),
 \]
which holds for all $0 < t < \min_j x_j$ (so that the appended component is the smallest one). Averaging \eqref{eq:Rad-diff-L3} over $X$ (which has finite support), and then using Schur-concavity \eqref{eq:L3-Schur}, we obtain
\begin{align*}
\E|X| &= \frac{1}{3}\lim_{t \to 0+} \frac{\E|X+\sqrt{t}\varepsilon|^3 - \E|X|^3}{t} \\
&\geq \frac{1}{3}\lim_{t \to 0+} \frac{\E|\sqrt{1+t}X|^3 - \E|X|^3}{t} \\
&= \frac{1}{2}\E|X|^3.
\end{align*}
We check directly that equality holds for $X = \frac{\varepsilon_1+\varepsilon_2}{\sqrt{2}}$. This finishes the proof.
\end{proof}

\subsection{End of proof.}

To finish the argument, we need to show \eqref{eq:Szarek-hard-case} for $x \in \R_+^n$ with $\sum x_j = 1$ and $x_j \leq \frac12$. Using first the $L_1-L_3$ bridge \eqref{eq:Rad-bridge} and then Schur-concavity \eqref{eq:L3-Schur}, we obtain
\[ 
\E|X| \geq \frac12 \E|X|^3 \geq \frac{1}{2}\E\left|\frac{\varepsilon_1+\varepsilon_2}{\sqrt{2}}\right|^3 = \frac{1}{\sqrt2},
 \]
since $x \prec (\frac12, \frac12, 0, \dots, 0)$.\hfill$\square$

This establishes a blueprint we shall now follow to present a similar proof of Ball's cube-slicing result \eqref{eq:Ball}, where instead of Rademacher sums, we need to handle spherical ones.

\section{Ball's inequality made simple}

Throughout this section, $\xi, \xi_1, \xi_2, \dots$ are i.i.d. random vectors uniform on the unit sphere $S^2$ in $\R^3$. For a fixed vector $x \in \R^3$, consider the function
\[ 
H(t) = \E|x + \sqrt{t}\xi|, \qquad t \geq 0.
 \]
Assume $x \neq 0$. Since the marginal distribution $\scal{x}{\xi}$ is uniform on $[-|x|, |x|]$, we can explicitly evaluate it,
\begin{align*}
H(t) &= \E\sqrt{|x+\sqrt{t}\xi|^2} = \E\sqrt{|x|^2 + 2\sqrt{t}\scal{x}{\xi} + t} \\
&= \frac{1}{2}\int_{-1}^1 \sqrt{|x|^2+2\sqrt{t}|x|u + t} \dd u \\
&= \begin{cases}
|x| + \frac{t}{3|x|}, & 0 \leq t \leq |x|^2, \\ \sqrt{t} + \frac{|x|^2}{3\sqrt{t}}, & t \geq |x|^2.
\end{cases}
 \end{align*}
Again, the function $H$ is affine on $[0, |x|^2]$, the one-sided derivatives at $t = |x|^2$ match and $H''(t) = -\frac{t-|x|^2}{4t^{5/2}}$, $t \geq |x|^2$, so $H$ is strictly concave on $(|x|^2, \infty)$ and concave globally. With hindsight, we also note the affine envelope bound,
\begin{equation}\label{eq:H-diff-bd}
|x| \leq H(t) \leq |x| + \frac{t}{3|x|}, \qquad t \geq 0,
\end{equation}
as well as the value of the right derivative
\begin{equation}\label{eq:H'(0)}
H'(0+) = \frac{1}{3|x|}.
\end{equation}
When $x = 0$, we simply have $H(t) = \sqrt{t}$.
Analogously to \eqref{eq:L3-Schur}, we derive the following claim,
\begin{equation}\label{eq:L1-Schur}
G(x_1, \dots, x_n) = \E\Big[|\sqrt{x_1}\xi_1 + \dots + \sqrt{x_n}\xi_n|\Big] \qquad \text{is Schur-convex.}
\end{equation}
(It is instrumental to use the fact that a sum of independent spherically symmetric random vectors is spherically symmetric; we leave out the details.)
This goes back to the work of K\"onig and Kwapie\'n on Khinchin-type inequalities \cite{KK}, which has been further delightfully treated in greater generality in all dimensions by Baernstein II and Culverhouse in \cite{BC}.

\subsection[The L-1-L1 bridge]{The $L_{-1}-L_1$ bridge.}

The next lemma is the analogue of Lemma \ref{lm:Rad-bridge}, providing the key link between the $L_{-1}$ and $L_1$ moments for spherical sums, with integrability aspects requiring a bit more care. In particular, $\E|X|^{-1} < \infty$ for every \emph{non}zero spherical sum $X$; this is carefully explained in the next subsection (see \eqref{eq:sph-n=2} below).

\begin{lemma}\label{lm:Ball-bridge}
For every sum $X = \sum_{j=1}^{n} a_j\xi_j$, $a_j \in \R$ (not all zero), we have
\begin{equation}\label{eq:Ball-bridge}
\E[|X|^{-1}]\cdot \E[|X|^2] \leq \frac{3}{2}\E|X|,
\end{equation}
which is sharp.
\end{lemma}

\begin{proof}[Proof of Lemma \ref{lm:Ball-bridge}]
First, we note that Taylor's formula gives the following analogue of \eqref{eq:Rad-diff},
\[ 
\frac{\dd }{\dd t}\Big|_{t=0+} \E f(x+\sqrt{t}\xi) = \frac{1}{6}\Delta f(x),
 \]
for a $C^2$ function $f\colon \R^3 \to \R$. We would like to apply it to the Euclidean norm function $f(x) = |x|$ since $\Delta |x| = 2|x|^{-1}$, $x \neq 0$, averaged over the law of $X$, but this requires care with regularity at the origin and differentiation under the expectation. To avoid these technicalities, we derive the desired identity directly for this function.

\bigskip\noindent
\textbf{Claim.} For every random vector $Y$ in $\R^3$ independent of $\xi$ with $\E|Y|^{-1} < \infty$ and $\E|Y| < \infty$, we have
\begin{equation}\label{eq:sph-diff}
\frac{\dd }{\dd t}\Big|_{t=0+} \E |Y+\sqrt{t}\xi| = \frac13\E|Y|^{-1}.
\end{equation}
Indeed, conditioning on the values of $Y$, and then averaging over $Y$ the function $H_x(t) = \E_\xi|x + \sqrt{t}\xi|$ analysed in \eqref{eq:H-diff-bd}-\eqref{eq:H'(0)}, we have
\begin{align*}
\lim_{t \to 0+}  \frac{\E |Y+\sqrt{t}\xi| - \E|Y|}{t} &= \lim_{t\to 0+} \E_Y\frac{H_Y(t) - |Y|}{t} \\
&= \E_Y\left[\lim_{t\to 0+} \frac{H_Y(t) - |Y|}{t}\right]= \E_Y H_Y'(0+) = \E_Y \frac{1}{3|Y|},
 \end{align*}
where we have used pointwise bound \eqref{eq:H-diff-bd} and Lebesgue's dominated convergence theorem to justify swapping the order of taking the limit and expectation, and then applied \eqref{eq:H'(0)}.

It remains to repeat the comparison from Lemma \ref{lm:Rad-bridge}, with the inequality reversed: invoke identity \eqref{eq:sph-diff} with $Y = X = \sum \sqrt{x_j}\xi_j$ (assuming without loss of generality $x_j >0$, $\sum x_j = 1$), spread out the new variable $t$ among the variables $x_j$, and finish with Schur-convexity  \eqref{eq:L1-Schur} of the $L_1$-norm to conclude $\frac13\E|X|^{-1} \leq \frac12\E|X|$. Again, it can be checked that equality holds for $X = \frac{\xi_1+\xi_2}{\sqrt2}$ (we have $\E|\xi_1+\xi_2| = \frac43$ and $\E|\xi_1+\xi_2|^{-1} = 1$).
\end{proof}

We now use the same dichotomy according to whether one coefficient
carries at least half of the total squared mass.

\subsection[Easy case]{\emph{Easy} case: $\max_j |a_j| \geq \frac{1}{\sqrt2}$.}

A spherical counterpart of \eqref{eq:L_1-n=2} in $\R^3$ reads as follows
\begin{equation}\label{eq:sph-n=2}
\E|a \xi_1 + b\xi_2|^{-1} = \min\{|a|^{-1}, |b|^{-1}\}, \qquad a, b \in \R 
\end{equation}
(with the convention $1/0 = +\infty$). It can be checked by a direct calculation akin to the one for function $H(t)$ (see, e.g. Lemma 6.6 in \cite{CNT}), or justified by means of the divergence theorem (see, e.g. Lemma 2 in \cite{GTW}). The direct calculation with $a, b > 0$ reads as follows
\[ 
\E|a\xi_1+b\xi_2|^{-1} = \frac12\int_{-1}^1 (a^2+b^2+2ab u)^{-1/2} \dd u = \frac{(a+b)-|a-b|}{2ab} = \frac{1}{\max\{a, b\}}.
 \]

Consequently, singling out a nonzero summand, say $a_1\xi_1$ in $X = \sum a_j\xi_j$, writing $X = a_1\xi_1 + X'$ and using the rotational invariance of the sum $X'$ of the remaining ones gives
\[ 
\E|X|^{-1} = \E\min\{|a_1|^{-1}, |X'|^{-1}\} \leq |a_1|^{-1}.
 \]
Incidentally, this explicitly establishes the integrability promised earlier.
(This is a probabilistic reincarnation of Ball's original geometric projection argument, see the beginning of the proof of Theorem 4 in \cite{Ball-cube}). Thus \eqref{eq:Ball} holds in the \emph{easy} case: $\max_j |a_j| \geq \frac{1}{\sqrt2}$.

\subsection[Hard case]{\emph{Hard} case: $\max_j |a_j| \leq \frac{1}{\sqrt2}$.}

Exactly as before, setting $x_j = a_j^2$, it remains to consider sums $X = \sum \sqrt{x_j}\xi_j$ with $\sum x_j = 1$ and $x_j \leq \frac12$ for every $j$. Starting with \eqref{eq:Ball-bridge}, and using the majorisation of such vectors $x$ by $(\frac12, \frac12, 0, \dots, 0)$ combined with \eqref{eq:L1-Schur} finishes the proof,
\[ 
\E|X|^{-1} \leq \frac{3}{2}\E|X| \leq \frac{3}{2}\E\left|\frac{\xi_1+\xi_2}{\sqrt2}\right| = \frac{3}{2\sqrt2}\frac43 = \sqrt2.
 \]

\section{Further remarks}

\subsection{Bridges without majorisation}

After discovering the unified proof mechanism for the inequalities of Szarek and Ball presented in the previous sections, we noticed an alternative way of obtaining the bridge inequalities \eqref{eq:Rad-bridge} and \eqref{eq:Ball-bridge}, the workhorses of the entire argument. The key insight is a replacement identity of Bara\'nski, Murawski, Nayar and Oleszkiewicz from Lemma~4 in \cite{BMNO}.

We illustrate the idea starting with Rademacher sums $X = \sum a_j\varepsilon_j$. Note that for a $C^1$ function $f\colon \R \to \R$ and $a, x \in \R$, we have
\begin{align*}
 \E[\varepsilon f(x+a\varepsilon)]
 &=\frac{f(x+a)-f(x-a)}2\\
 &=\frac12\int_{-1}^1\frac{\dd}{\dd u}f(x+au)\,\dd u
 =a\E[f'(x+aU)],
\end{align*}
where $U$ is a uniform random variable on $[-1,1]$. If we then take $U_1, U_2, \dots$ to be i.i.d. copies of $U$, independent of the Rademacher variables $\varepsilon_1, \varepsilon_2, \dots$, and set
\[ 
X^{(k)} = a_kU_k + \sum_{j \neq k} a_j\varepsilon_j,
 \]
then conditioning on the remaining signs, applying the above identity with $a=a_k$, multiplying by $a_k$ and summing over $k$ gives the Bara\'nski--Murawski--Nayar--Oleszkiewicz identity
\begin{equation}\label{eq:ole-trick}
\E\big[ Xf(X) \big] = \sum_{k=1}^n a_k^2\E\Big[f'(X^{(k)})\Big].
\end{equation}
Specialised to $f(x) = x|x|$, this yields 
\[ 
\E\big[|X|^3\big] = 2\sum_{k=1}^n a_k^2\E\Big[|X^{(k)}|\Big].
 \]
To obtain the desired bridge inequality \eqref{eq:Rad-bridge}, observe that for a convex function $\varphi\colon \R \to \R$, we have a simple uniform-to-Rademacher replacement bound
\begin{align*}
\E\big[\varphi(x + aU)\big] &\leq \E\Big[\varphi(x+a)\frac{1+U}{2} + \varphi(x-a)\frac{1-U}{2}\Big] \\
&= \frac{\varphi(x+a) + \varphi(x-a)}{2} = \E\big[\varphi(x+a\varepsilon)\big].
 \end{align*}
Thus, $\E|X^{(k)}| \leq \E|X|$ and \eqref{eq:Rad-bridge} follows.

There is a natural spherical counterpart of \eqref{eq:ole-trick}. Let $X=\sum_j a_j\xi_j$ be a nonzero spherical sum and discard zero coefficients. If $X^{(k)}$ is obtained by replacing $\xi_k$ by a random point $V_k$, uniform in the solid unit Euclidean ball $B_2^3$ and independent of all the sphere vectors, say
\[ 
X^{(k)} = a_kV_k + \underbrace{\sum_{j \neq k} a_j\xi_j}_{X'_k},
 \]
then the divergence theorem applied to a $C^2$ function $f\colon\R^3\to\R$ gives
\begin{equation}\label{eq:ole-trick-R3}
\E\scal{X}{\nabla f(X)} = \frac{1}{3}\sum_{k=1}^n a_k^2 \E\big[\Delta f(X^{(k)}) \big].
\end{equation}
Applying this first to $f_\delta(x) = \sqrt{|x|^2+\delta}$, $\delta > 0$ and then letting $\delta \to 0+$ yields
\[ 
\E|X| = \frac{2}{3}\sum_k a_k^2 \E\Big[|X^{(k)}|^{-1}\Big].
 \]
Taking the limit in expectation is justified on the left by monotone convergence: the integrand is $\scal{X}{\nabla f_\delta(X)} = \frac{|X|^2}{\sqrt{|X|^2+\delta}} \uparrow |X|$, and on the right by the pointwise bound 
\[
0 \leq \Delta f_\delta(x) = \frac{3}{\sqrt{|x|^2+\delta}} - \frac{|x|^2}{(|x|^2+\delta)^{3/2}} \leq \frac{3}{|x|},
\] 
combined with the integrability 
\[
\E|X^{(k)}|^{-1} \leq |a_k|^{-1}\E|V_k|^{-1} = \frac32 |a_k|^{-1},
\]
using \eqref{eq:sph-n=2} and the fact that $|V_k|$ has density $3r^2$ on $(0,1)$. Moreover, a replacement principle for the Newton potential $|x|^{-1}$ also follows from \eqref{eq:sph-n=2},
\begin{align*}
\E\Big[|X^{(k)}|^{-1}\Big] &= \E\min\{|a_k|^{-1}|V_k|^{-1}, |X'_k|^{-1}\} \\
&\geq  \E\min\{|a_k|^{-1}|\xi_k|^{-1}, |X'_k|^{-1}\} = \E\Big[|X|^{-1}\Big],
\end{align*}
since $|V_k| \leq 1 = |\xi_k|$, pointwise. Consequently, \eqref{eq:ole-trick-R3} also proves  the spherical bridge inequality \eqref{eq:Ball-bridge}, without the majorisation trick.

\subsection{Szarek made even simpler}

The distinction between the easy and hard cases in the proof of Szarek's inequality can be dispensed with by virtue of H\"older's inequality. Indeed,
\[ 
\E|X|^2 \leq \sqrt{\E|X|}\sqrt{\E|X|^3},
 \]
which, for a Rademacher sum $X$ with $\E|X|^2 = 1$, combined with the bridge inequality \eqref{eq:Rad-bridge}, $\E|X|^3\leq2\E|X|$, immediately yields
\[ 
1 \leq \sqrt{\E|X|}\sqrt{2\E|X|} = \sqrt{2}\E|X|,
 \]
that is, Szarek's inequality \eqref{eq:Szarek} (again, without invoking majorisation used in the \emph{hard} case).

This shortcut, however, does not give Ball's sharp constant: the spherical bridge followed by $\E|X|\leq(\E|X|^2)^{1/2}=1$ gives only $\E|X|^{-1}\leq3/2$. In the Rademacher case, the H\"older step is sharp at the extremiser because $|X|$ is supported on $\{0,\sqrt2\}$; the extremal spherical radius is not constant. The split into two coefficient regimes is therefore natural for Ball's problem, and it is precisely the Schur comparison of the lifted first moment in the small-coefficient regime (the \emph{hard} case) that allows us to use the bridge effectively. This is why we have retained the original presentation for Szarek's inequality and have chosen to show a common mechanism for the two inequalities, rather than concealing it behind a shortcut particular to Rademachers.

\subsection{A bit of history and wider context}

Moment comparison inequalities are most interesting as reversals of the plain monotonicity $\|X\|_p\leq\|X\|_q$, $0<p\leq q$, of the $L_p$-norms\footnote{As customary, we retain the terminology of a norm when $0<p<1$.} $\|X\|_p=(\E|X|^p)^{1/p}$. For Rademacher sums, such comparisons go back to Khinchin's century-old 1923 work \cite{Khinchin} on dyadic expansions. The conjecture that the sharp constant in \eqref{eq:Szarek} is $1/\sqrt2$, equivalently that we have $\|X\|_2 \leq \sqrt2\|X\|_1$ for every Rademacher sum $X=\sum_j a_j\varepsilon_j$, was posed by Littlewood in his 1930 paper \cite{Littlewood}, in relation to questions on boundedness of bilinear forms in infinitely many variables. Littlewood's conjecture was settled positively by Szarek in his master's thesis and published in 1976, \cite{Sz}. 

Szarek's argument is inductive and involves a delicate analysis of the coefficient configurations. It is organised around the homogeneous ratio
$F(a)=\frac{1}{|a|}\E|\sum_j a_j\varepsilon_j|.$
The proof exploits the piecewise-linear structure of the numerator: by \cite[Lemma~1]{Sz}, at a local minimum of $F$, the sign vectors producing exact cancellation must span $a^\perp$. The corollary to that lemma yields a quantitative starting estimate in six dimensions, while \cite[Lemma~2]{Sz} provides a local estimate near the two-coordinate extremiser using the lower-dimensional inequality. The induction then distinguishes vectors near this extremiser from those bounded away from it, rather than splitting according to the largest coefficient, and propagates a stronger bound in the latter region. A four-to-two coefficient reduction averages the first moment exactly but can decrease the average Euclidean norm. The infinite-product estimate of \cite[Lemma~3]{Sz} controls the accumulated loss and shows that the initial gap above $1/\sqrt2$ survives away from the extremiser in every dimension.

Tomaszewski \cite{Tomas87} subsequently gave a shorter proof based on an elementary but quite sophisticated combinatorial lemma tailored to the $L_1$--$L_2$ comparison (the lemma is proved by merging two coefficients, and the theorem
follows by summing over partitions of the coefficients and choices
of signs).

A completely different approach had been introduced by Haagerup in \cite{Haa} in his
work completing the determination of \emph{all} the sharp $L_p$--$L_2$ constants for Rademacher sums $X$. The main idea is analytic to its core, relying on Fourier representations of power functions $|x|^p$, thus allowing one to take advantage of the independence of the summands $a_j\varepsilon_j$ on the Fourier side by studying univariate integrals of their characteristic functions, however oscillatory in nature. Concretely, to give a flavour, for $X=\sum_j a_j\varepsilon_j$, we have the identity
\begin{equation}\label{eq:Haagerup-formula}
 \mathbb{E}|X|
 =\frac2\pi\int_0^\infty
       \frac{1-\prod_j\cos(a_jt)}{t^2}\,\dd t.
\end{equation}
Haagerup's next ingenious idea was to bound the product by the product of its absolute values and then \emph{decouple} it using the weighted AM--GM inequality (with weights related to the coefficients $a_j$), turning the problem into rather subtle estimates on certain special functions. Remarkably, this step is tight at the two-coordinate extremiser. Here decoupling means an elementary separation of a product; it should be distinguished from decoupling theory in harmonic analysis, with its far-reaching applications to number theory and geometry. For accounts of the latter, see \cite{Pierce, Guth}.

Lata\l a and Oleszkiewicz in their early work \cite{LO} found another approach to Szarek's inequality. Their striking proof (nota bene, very short) in modern language is based on the discrete Fourier analysis on the Boolean cube $\{-1,1\}^n$ (Walsh functions), and does much more than establish the scalar inequality \eqref{eq:Szarek}. See, e.g. \cite{NT-surv} for a streamlined presentation. For arbitrary vectors $x_1, \dots, x_n$ in any normed space, they show 
\[ 
\left(\E\left\| \textstyle\sum_j \varepsilon_jx_j \right\|^2\right)^{1/2} \leq \sqrt2 \cdot \E\left\| \sum_j \varepsilon_jx_j \right\|.
 \]
No Hilbert-space structure or geometric assumption on the norm is required (in fact \eqref{eq:Szarek} extends \emph{quasi gratis} to the Hilbert-space setting, see the so-called Szarek-Orlicz argument, Remark 3 in \cite{Sz}). We emphasise that the second moment above is not replaced by $\sum_j \|x_j\|^2$, which would be a very different assertion outside Hilbert space. This is the sharp $L_1-L_2$ case of the vector-valued moment comparisons initiated by Kahane in \cite{Kahane}, see also \cite{LT} for their role in probability in Banach spaces. The essence of the Lata\l a-Oleszkiewicz argument, in the later operator formulation \cite{KLO, Olesz}, is a gain in the spectral gap for the discrete Laplacian coming from symmetry, combined with an effective use of convexity of the norm (a pointwise bound), resulting in a Poincar\'e-type inequality. This exceptionally economical and powerful mechanism is one of the main motivations behind this paper, where we sought an equally transparent approach to Ball's inequality in the framework of spherical harmonics. The challenge was to connect this spectral framework with the negative moment in Ball's problem. Although this did not yield a spectral proof, it led to the moment-lifting argument underpinning this work.

By now, there is a substantial literature on Khinchin-type inequalities for both Rademacher sums \cite{Haa, Jakim, KLO, LO}, as well as other distributions: uniform variables on intervals \cite{LO-uniform, CGT} and finite sets \cite{HT}, Steinhaus variables \cite{Sawa, Konig-Steinhaus, RTW} and higher-dimensional spherical analogues \cite{BC, KK, JTT}, Gaussian mixtures \cite{ENT-mixtures}, marginals of $\ell_p$-balls \cite{ENT-moments}, type-L variables \cite{Newman, HNT}, or exponentials \cite{BTT}. These comparisons have strong connections to sharp bounds for sections and projections \cite{BGMN, BN, CKT, KRZ, MP, NT-surv}. Comparisons between moments other than the second continue to stimulate new approaches; the treatment of even moments via ultra log-concavity \cite{NO}, and Hadamard factorisation \cite{HNT}, souped-up to more precise Schur-monotonicity results \cite{Chavez}, as well as the recent $L_p$--$L_4$ theorem of \cite{BMNO}, provide examples.

Ball's proof is a landmark in the Fourier-analytic treatment of section problems, and to the best of our knowledge is the only available one prior to this paper. Akin to \eqref{eq:Haagerup-formula}, the Fourier-inversion formula applied to the density of the relevant sum of uniform random variables gives
\begin{equation}\label{eq:Ball-formula}
\vol_{n-1}([-\tfrac12,\tfrac12]^n \cap a^\perp) = \frac2\pi\int_0^\infty\prod_j\frac{\sin(a_jt)}{a_jt}\ \dd t,
\end{equation}
for every unit vector $a$ in $\R^n$.

Then Ball's argument proceeds along very similar lines to Haagerup's, with a decoupling step taken care of by means of H\"older's inequality, which reduces the problem in the small-coefficient regime (the \emph{hard} case) to showing the following integral inequality:
\[
\frac1\pi\int_{\R}\left|\frac{\sin t}{t}\right|^s\ \dd t 
\leq
\frac1\pi\int_{\R}\left(e^{-t^2/(2\pi)}\right)^s\ \dd t =\sqrt{\frac2s}, \qquad s\geq2.
\]
This is a rather delicate estimate due to the oscillatory nature of the integrand. It can be viewed as the key analytic component and workhorse of the whole approach. Our approach here bypasses that entirely. Ball's direct proof of the integral inequality above relied to a large extent on several carefully crafted pointwise estimates coming from power series expansions and comparisons with the Gaussian function. Later, Nazarov and Podkorytov in \cite{NP} offered an alternative, robust and efficient method based on controlling the number of intersections between the distribution functions of $|\frac{\sin t}{t}|$ and $e^{-t^2/(2\pi)}$. More recently, Melbourne and Roberto in \cite{MRtransport} put forward yet another point of view and developed an efficient mechanism in terms of controlling the derivative of relevant transport maps. 
The Fourier representation \eqref{eq:Ball-formula} goes back to P\'olya \cite{Polya}; in the context of cube slicing, it can be traced back to Hensley's work~\cite{Hensley}, where it was used as a proxy in his pioneering probabilistic viewpoint. Fourier techniques subsequently grew into a substantial part of convex geometry, developed systematically in Koldobsky's work \cite{Kold-sections} and his monograph \cite{Kold-book}. A prominent example is the resolution of the Busemann-Petty problem (Problem 1 in \cite{BP}), culminating in \cite{GKS}.

We finish by reiterating that Khinchin-type inequalities provide a natural probabilistic language for both section and projection questions, well beyond the present two problems. This paradigm is central in the recent survey \cite{NT-surv}. Here, concretely, for a unit vector $a$ in $\R^n$,
\begin{align}
 \frac{\vol_{n-1}(B_\infty^n \cap a^\perp)}{\vol_{n-1}(B_\infty^{n-1})}
 &=\E\left|\sum_j a_j\xi_j\right|^{-1},
     \label{eq:remarks-section-formula}\\
 \frac{\vol_{n-1}(P_{a^\perp}B_1^n)}{\vol_{n-1}(B_1^{n-1})}
 &=
      \E\left|\sum_j a_j\varepsilon_j\right|,
    \label{eq:remarks-projection-formula}
\end{align}
and related representations have been established for normalised sections of $\ell_p$-balls $B_p^n=\{x\in\R^n, \ \|x\|_p\leq1\}$ for $0<p\leq\infty$, and for their hyperplane projections in the convex range $1\leq p\leq\infty$; see \cite{BGMN, BN, NT-surv}. The first identity can be traced back to \cite{KK}. The second, recorded in \cite{Ball-wave}, follows by projecting the facets of the cross-polytope $B_1^n$; for the corresponding use of Cauchy's projection formula for other $\ell_p$ balls, see \cite{BN}. It is still intriguing that the negative moment of a spherical sum describes sections, while the positive moment of a Rademacher sum describes projections, for the polar pair $(B_\infty^n)^\circ=B_1^n$.

\subsection{Future work}\label{sec:future}

\paragraph{Stability.}

Once a sharp inequality and its equality cases are known, the next natural question is whether near equality forces the coefficients to be close to an extremiser. For Szarek's inequality, see \cite{DDS}; for Ball's inequality, see \cite{CNT, MRstability}, with explicit estimates for both in \cite{ENT}. For higher moments, dimensions and other aspects, see also \cite{Chavez, ENT-distribution, Jakim, Jakim2, JTT} The analogous question for polydisc slicing is treated in \cite{GTW}. We plan to retain the losses in our lifting and majorisation arguments to give direct, more transparent proofs of such stability estimates in the companion paper \cite{TT}, with better numerical constants.

\paragraph{Higher dimensions.}

The differentiation principles \eqref{eq:Rad-diff} and \eqref{eq:sph-diff} extend \emph{mutatis mutandis} to all dimensions: for $\xi$ uniform on the unit sphere $S^{d-1}$,
\[
 \left.\frac{\dd}{\dd t}\right|_{t=0+}
 \E f(x+\sqrt t\,\xi)=\frac1{2d}\Delta f(x).
\]
Here $f$ is a $C^2$ function near $x$.
Together with the spherical majorisation principle of \cite{BC}, we shall provide a common treatment of the known sharp slicing inequalities for unit balls in $\ell_\infty(\ell_2)$ by block subspaces, in particular bypassing the difficult Fourier estimates in the works of Oleszkiewicz and Pe\l czy\'nski \cite{OP} and Brzezinski \cite{Brz}. This method also gives sharp $L_p$--$L_2$ Khinchin inequalities for sums of independent random vectors uniform on $S^{d-1}$, for every $d\geq5$ and $-(d-2)\leq p\leq-2$. We will present these results in \cite{TT}, extending Theorem~1 from \cite{CST} into part of the range mentioned as open in Table 1 therein.
Further aspects, including moment comparisons and quantitative questions in this setting, have been also studied in \cite{CST, Eskenazis-block, JTT}.  Moreover, very recently, stability of Brzezinski's bound has been established in \cite{BRT}, albeit without explicit constants. Our companion paper will rectify that.

\paragraph{Other bodies.}

A different stability question concerns perturbations of the summand distribution; see \cite{ENT-distribution}. In particular, the resilience theorem of \cite{ENT} says that the two-coordinate normal continues to maximise sections of real $\ell_p$ balls for all sufficiently large $p$, uniformly in the dimension, and to minimise projections of $\ell_q$ balls for $q$ sufficiently close to $1$. It would be interesting to find a lifting adapted to these perturbed laws which remains sharp at the two-coordinate sum, and thereby obtain a simpler proof of resilience.

The simplex gives another natural test. Webb's upper bound \cite{Webb} already has a proof based on comparisons of half-line moments of log-concave densities; see \cite[Lemma~1.7]{Webb-thesis}. This fits the bridge philosophy, although it is not the same replacement identity. The negative-moment point of view is developed further in \cite{MRTT}. For the opposite, lower bound, Ambrus and G\'argy\'an \cite{AG} recently proved that the minimal central hyperplane sections are parallel to facets, following the asymptotically sharp result of Tang \cite{Tang-simplex}. We would like to understand whether lifting and convexity can give an alternative proof of this lower bound.

\paragraph{Acknowledgments.} The compelling idea of using majorisation has been floating around for a long time. We are particularly grateful to Peter Pivovarov and Paul Simanjuntak for a brief but very inspiring discussion back in 2020 at the University of Missouri in Columbia (as well as the lovely visit there), and to Piotr Nayar for the many discussions.

\paragraph{AI disclosure.} 
During preliminary work, T.T. explored a spectral proof of Ball’s cube slicing inequality in discussions with ChatGPT, which contributed to the moment-lifting approach and the initial arguments. The authors subsequently rederived, simplified, and developed the proofs and wrote the final exposition.

\end{document}